\documentclass[11pt,reqno]{amsart}
\usepackage{amssymb,amsmath,mathrsfs,enumerate,color,tikz,hyperref,cite}
\usepackage[top=2cm, left=1.8cm, a4paper]{geometry}

\newtheorem{thm}{Theorem}
\newtheorem{proposition}[thm]{Proposition}
\newtheorem{lemma}[thm]{Lemma}
\newtheorem{fact}[thm]{Fact}
\newtheorem{cor}[thm]{Corollary}
\newtheorem{question}{Question}

\theoremstyle{definition}
\newtheorem{definition}[thm]{Definition}
\newtheorem{remark}[thm]{Remark}

\theoremstyle{remark}
\newtheorem{claim}{Claim}

\def\Lip{\operatorname{Lip}}
\def\er{\mathbb R}
\def\dist{\operatorname{dist}}
\def\setsep{;\;}
\def\F{\mathcal F}
\def\en{\mathbb N}
\def\zet{\mathbb Z}

\def \sspan {\operatorname{span}}
\def \Br {\operatorname{Br}}

\usetikzlibrary{matrix}

\begin{document}
\title{Lipschitz-free spaces over ultrametric spaces}
\author{Marek C\' uth}
\author{Michal Doucha}
\email{marek.cuth@gmail.com, m.doucha@post.cz}
\address[M.~C\' uth, M.~Doucha]{Instytut Matematyczny Polskiej Akademii Nauk, \' Sniadeckich 8, 00-656 Warszawa, Poland}

\begin{abstract}We prove that the Lipschitz-free space over a separable ultrametric space has a monotone Schauder basis and is isomorphic to $\ell_1$. This extends results of A. Dalet using an alternative approach.
\end{abstract}
\maketitle
\section*{Introduction}

Let $(M,d,0)$ be a pointed metric space, that is, a metric space equipped with a distinguished
point denoted by 0. To such a space we can associate the space $\Lip_0(M)$ of all real-valued Lipschitz functions $f$ on M which satisfy $f(0) = 0$, endowed with the norm $\|\cdot\|_{Lip}$ defined by the Lipschitz constant, i.e.
$$\|f\|_{Lip} : = \sup\left\{\frac{|f(x) - f(y)|}{d(x,y)}\setsep x,y\in M,\;x\neq y\right\}.$$
It is readily checked to be a Banach space.

The Dirac map $\delta_M:M\to \Lip_0(M)^*$ defined by $\delta_M(x)(f) = f(x)$ for $x\in M$ and $f\in \Lip_0(M)$ is an isometric embedding from $M$ into $\Lip_0(M)^*$. The \emph{Lipschitz-free space over $M$}, denoted by $\F(M)$, is the closed linear hull of $\delta_M(M)$ in $\Lip_0(M)^*$, i.e. $\F(M): = \overline{\sspan}\{\delta_M(x)\setsep x\in M\}$. It is known that its dual space is isometrically isomorphic to $\Lip_0(M)$. We refer to \cite{weaver} and \cite{godKal} for an introduction to Lipschitz-free spaces and its basic properties.

The study of the linear structure of Lipschitz-free spaces over metric spaces has become an active field of study, see e.g. \cite{godard, godKal, lanPer, hajPer, godOza, aude1, aude2}. In this note we are interested in the structure of Lipschitz-free spaces over ultrametric spaces. Let us recall that a  metric space $(M, d)$ is said to be \emph{ultrametric} if for every $x, y, z \in M$, we have $d(x, z) \leq \max\{d(x, y), d(y, z)\}$.

A. Dalet \cite{aude2} proved, among other things, that the Lipschitz-free space over a separable proper ultrametric space has the metric approximation property and is isomorphic to $\ell_1$. We improve the result and show the following.

\begin{thm}\label{t:schauderUltrametric}The Lipschitz-free space over a separable ultrametric space has a monotone Schauder basis.
\end{thm}

\begin{thm}\label{t:freeUltrametric}The Lipschitz-free space over a separable ultrametric space is isomorphic to $\ell_1$.\end{thm}

\noindent The improvement is that we do not assume the ultrametric space to be proper. Moreover, in Theorem \ref{t:schauderUltrametric} we get a stronger conclusion. Our proofs do not follow the lines of the proofs from \cite{aude2} and so they can be viewed as an alternative approach to the the above mentioned results of A. Dalet as well. In the final section we collect few corollaries of our results and suggest some open problems.

Before coming to the proofs, let us recall some basic results. One of the main properties of the Lipschitz-free spaces is the following universality property that provides a connection between the Lipschitz maps in metric spaces and linear maps in Banach spaces; see \cite[Lemma 2.5]{godKal}.

\begin{lemma}\label{l:basicFreeSpace}Let $M$ be a pointed metric space and $X$ a Banach space and suppose $L: M\to X$ is a Lipschitz map such that $L(0_M) = 0$. Then there exists a unique linear map $\widehat{L}:\F(M)\to X$ extending $L$, i.e. the following diagram commutes:
\begin{center}\begin{tikzpicture}
  \matrix (m) [matrix of math nodes,row sep=3em,column sep=4em,minimum width=2em]
  {
    M & X \\
   \F(M) & X \\};
  \path[-stealth]
    (m-1-1) edge node [left] {$\delta_M$} (m-2-1)
            edge node [above] {$L$} (m-1-2)
    (m-2-1.east|-m-2-2) edge [dashed] node [above] {$\widehat{L}$} (m-2-2)
    (m-1-2) edge node [right] {$\mathrm{id}_X$} (m-2-2);    
\end{tikzpicture}\end{center}
and $\|\widehat{L}\| = \|L\|_{Lip}$ where $\|\cdot\|_{Lip}$ denotes the Lipschitz norm of $L$.

\end{lemma}

\noindent Note that Lemma 2.5 in \cite{godKal} is formulated only for the case when $M$ is a Banach space; however, a similar proof works also in the more general setting of Lemma \ref{l:basicFreeSpace}. Moreover, it is possible to prove Lemma \ref{l:basicFreeSpace} directly in a similar way as \cite[Lemma 2.2]{godKal} - it is enough to replace $\Lip_0(Y)$ by $Y^*$ in its proof.

Note that it is straightforward to check that for Lipschitz maps $L:M\to N\subset \F(N)$ and S:$N\to P\subset\F(P)$ with $L(0_M)=0_N$ and $S(0_N) = 0_P$ we have $\widehat{SL} = \widehat{S}\widehat{L}$. Hence, as an immediate consequence of Lemma \ref{l:basicFreeSpace} we get the following facts, which we will use later.

\begin{fact}\label{f:staciRetrakce}Let $(M,d)$ be a metric space, $K > 0$ and $A\subset M$ be a $K$-Lipschitz retract of $M$. Then there exists a norm-$K$ projection from $\F(K)$ onto $\F(A)$; i.e. $\F(A)$ is a $K$-complemented subspace of $\F(M)$.
\end{fact}

\begin{fact}\label{f:isomorphism}Let $M, N$ be $K$-bi-Lipschitz equivalent metric spaces for some $K > 0$. Then $\F(M)$ is $K$-isomorphic to $\F(N)$.
\end{fact}

\noindent Let $(M,d)$ be an ultrametric space. We will often use the following property of ultrametric spaces which is easy to prove. For $x,y,z\in M$, if $d(x,y)\neq d(y,z)$ then $d(x,z) = \max\{d(x,y),d(y,z)\}$.

\section{Monotone Schauder basis}\label{s:monSchauder}

The purpose of this section is to prove Theorem \ref{t:schauderUltrametric}.

\begin{lemma}\label{l:sufficientSchauderBasis}Let $(M,d)$ be a separable pointed metric space. Let $(s_n)_{n\in\en}$ be a one-to-one sequence of points from $M$ with $0_M = s_1$ and $\overline{\{s_n\setsep n\in\en\}} = M$. Let there exist a system of retractions $(r_n)_{n\in\en}$ such that, for every $n\in\en$, we have
	\begin{enumerate}[\upshape (i)]
		\item $r_n$ is a 1-Lipschitz retraction with $r_n(M) = \{s_k\setsep k\leq n\}$, and
		\item $r_n\circ r_{n+1} = r_n$.
	\end{enumerate}
Then $\F(M)$ has a monotone Schauder basis.
\end{lemma}
\begin{proof}Since $r_n(M)\subset r_{n+1}(M)$ for every $n\in\en$, we have $r_{n+1}\circ r_n = r_n$. By Lemma \ref{l:basicFreeSpace}, there are projections $P_n:\F(M)\to \F(M)$ with $\|P_n\|\leq 1$,  $P_n(\F(M)) = \sspan\{\delta_M(s_k)\setsep k\leq n\}$ and $P_n\circ P_m = P_{\min\{n,m\}}$ for every $n,m\in\en$. Obviously, $\dim P_n(\F(M)) = n$. Since $\bigcup_{n\in\en} P_n(\F(M))$ is dense in $\F(M)$, we have $P_n(x)\to x$ for every $x\in\F(M)$. 
Now, it is a folklore fact \cite[Lemma 4.7]{FHHMZ} that such a system of projections gives us a monotone Schauder basis on $\F(M)$.
\end{proof}

\begin{proof}[Proof of Theorem \ref{t:schauderUltrametric}]Let $(M,d)$ be a separable ultrametric space and fix a one-to-one sequence $(s_n)_{n\in\en}$ of points from $M$ with $0_M = s_1$ and $\overline{\{s_n\setsep n\in\en\}} = M$. For every $n\in\en$, put $S_n: = \{s_k\setsep k\leq n\}$. We will find a sequence of retractions $(r_n)_{n\in\en}$ satisfying the assumptions of Lemma \ref{l:sufficientSchauderBasis}.

Fix $n\in\en$.  First, we put $I_n(x):=\{k\in\en\setsep k\leq n\text{ and }\dist(x,S_n) = d(x,s_k)\}$. We denote by $i_n(x)$ the minimal natural number from $I_n(x)$. Finally, we define $r_n:M\to S_n$ by 
$$r_n(x): = s_{i_n(x)},\quad x\in M.$$
Now, we will verify that the sequence $(r_n)_{n\in\en}$ meets the requirements (i) and (ii) from Lemma \ref{l:sufficientSchauderBasis}. First, observe the following.

\begin{claim}\label{claim:1}
\begin{equation}\label{eq:sameIndices}
\forall x,y\in M\quad d(x,y) < \dist(x,S_n)\Rightarrow i_n(x) = i_n(y).
\end{equation}
\end{claim}
\begin{proof}
Fix $x,y\in M$ with $d(x,y) < \dist(x,S_n)$. In order to see that \eqref{eq:sameIndices} holds, we show $\dist(y,S_n) = \dist(x,S_n)$. Indeed,
$$\dist(y,S_n)\leq d(y,s_{i_n(x)})\leq \max\{d(y,x), d(x,s_{i_n(x)})\} =  d(x,s_{i_n(x)}) = \dist(x,S_n).$$
Thus, in order to get a contradiction let us assume $\dist(y,S_n) < \dist(x,S_n)$. Then
$$\dist(y,S_n) < \dist(x,S_n) \leq d(x,s_{i_n(y)}) \leq \max\{d(x,y), d(y,s_{i_n(y)})\} = \max\{d(x,y), \dist (y,S_n)\}.$$
Now, if $\max\{d(x,y), \dist (y,S_n)\} = \dist (y,S_n)$, we get $\dist(y,S_n) < \dist(y,S_n)$, a contradiction. Otherwise, $\max\{d(x,y), \dist (y,S_n)\} = d(x,y) < \dist(x,S_n)$ and we get $\dist(x,S_n) < \dist(x,S_n)$, a contradiction. Thus, \eqref{eq:sameIndices} holds.
\end{proof}

Fix $x,y\in M$. In order to see that $r_n$ is a 1-Lipschitz mapping, we need to verify
\begin{equation}\label{eq:lipschitz}
\forall x,y\in M\quad d(s_{i_n(x)},s_{i_n(y)}) \leq d(x,y).
\end{equation}
If $d(x,y) < \max\{\dist(x,S_n),\dist(y,S_n)\}$ we get from \eqref{eq:sameIndices} and the symmetry of the situation that $i_n(x) = i_n(y)$ and \eqref{eq:lipschitz} is obvious. On the other hand, if $d(x,y) \geq \max\{\dist(x,S_n),\dist(y,S_n)\}$ we get
$$d(s_{i_n(x)},s_{i_n(y)})\leq \max\{d(s_{i_n(x)},x), d(x,y), d(y,s_{i_n(y)})\} = \max\{\dist(x,S_n), d(x,y), \dist(y,S_n)\}\leq d(x,y)$$
and \eqref{eq:lipschitz} holds.

It remains to show that, for every $n\in\en$, we have $r_n\circ r_{n+1} = r_n$. Fix $n\in\en$ and $x\in M$. Then it follows from the definitions above that either $i_{n+1}(x) = i_n(x)$ or $i_{n+1}(x) = n+1$. In both cases we will get $i_n(s_{i_{n+1}(x)}) = i_n(x)$. Indeed, this is trivial in the first case. Assume $i_{n+1}(x) = n+1$. Then $d(x,s_{n+1}) < \dist(x, S_n)$ and it follows from Claim \ref{claim:1} that $i_n(x) = i_n(s_{n+1})$. Hence, $i_n(s_{i_{n+1}(x)}) = i_n(s_{n+1}) = i_n(x)$. Therefore,
$$r_n(r_{n+1}(x)) = r_n(s_{i_{n+1}(x)}) = s_{i_n(s_{i_{n+1}(x)})} = s_{i_n(x)} = r_n(x)$$
and we are done.
\end{proof}

\section{Isomorphism with $\ell_1$}\label{s:ell1}

The purpose of this section is to prove Theorem \ref{t:freeUltrametric}. First, let us recall the notion of $\er$-trees and its link with Lipschitz-free spaces and ultrametric spaces.

\begin{definition}Let $(T,d)$ be a metric space such that, for every $x,y\in T$, there exists a unique isometry $\phi_{x,y}:[0,d(x,y)]\to T$ with $\phi_{x,y}(0) = x$ and $\phi_{x,y}(d(x,y)) = y$. Then we say that $T$ is an \emph{$\er$-tree} and we define the segment $[x,y]$ by $[x,y]: = \phi_{x,y}([0,d(x,y)])$.

Moreover, we say that $v\in T$ is a \emph{branching point} of $T$ if there are three points $x_1,x_2,x_3\in T\setminus\{v\}$ such that $[x_i,v]\cap[x_j,v] = \{v\}$ whenever $i,j\in\{1,2,3\}$, $i\neq j$. We denote by $\Br(T)$ the set of branching points of $T$.
\end{definition}

The link with Lipschitz-free spaces is contained in the following result, which has been proved by Godard in \cite[Corollary 3.4]{godard}. Note that the definition of an $\er$-tree and of a branching point above is not exactly as in \cite{godard}, but it is equivalent to it; see \cite[Chapter 3]{evans}.

\begin{proposition}\label{p:godard}Let $T$ be a separable $\er$-tree, and $A$ an infinite subset of $T$ such that $\Br(T)\subset\overline{A}$. If $\overline{A}$ does not contain any segment $[x,y]$ for $x\neq y$, then $\F(A)$ is isometric to $\ell_1$.
\end{proposition}

It belongs to a folklore fact that every ultrametric space embeds into an $\er$-tree. The shortest way of proving this statement is probably to show that an ultrametric space satisfies the ``four-point condition'' and then use the well-known fact that every metric space which satisfies this condition isometrically embeds into an $\er$-tree, see e.g. \cite[Theorem 3.38]{evans}.

In the following we will combine those two links and show that the Lipschitz-free space over a separable ultrametric space is isomorphic to $\ell_1$. First, we observe that it is enough to consider only ``$2^n$-valued ultrametric spaces''.

\begin{definition}A metric is said to be $2^ n$-valued if the only values assumed by the metric are $2^n$, $n\in\zet$.
\end{definition}

\begin{fact}\label{f:distances2}Any ultrametric space is $2$-bi-Lipschitz equivalent to a $2^n$-valued ultrametric space.
\end{fact}
\begin{proof}Let $(M,d)$ be an ultrametric space. We put $\rho(x,y): = 2^n$ whenever $d(x,y)\in[2^n,2^{n+1})$. Then it is easy to see that $\rho:M\times M\to [0,\infty)$ is an ultrametric on $M$ and $\rho(x,y)\leq d(x,y)<2\rho(x,y)$.
\end{proof}

Next, we show that the embedding of an ultrametric space into $\er$-tree may be done in such a way that it satisfies certain additional conditions, see Proposition \ref{p:ultrametricInTree}. In order to find an $\er$-tree into which our ultrametric spaces embeds, we will follow the ideas from \cite[Theorem 3.38]{evans}, where it is proved that any metric space satisfying the ``four-point condition'' embeds isometrically into an $\er$-tree. However, our space is an ultrametric space, so the construction will be done in an easier way. Once the construction is done, we will show that the additional conditions mentioned above are satisfied.

We begin with the following Lemma, which is inspired by \cite[Lemma 3.5]{evans}.

\begin{lemma}\label{l:tree}Let $(M,d)$ be a metric space and let $(\phi_{x,y})_{x,y\in M}$ be a a family of isometries such that $\phi_{x,y}:[0,d(x,y)]\to M$ is an isometry with $\phi_{x,y}(0) = x$ and $\phi_{x,y}(d(x,y)) = y$. Put $[x,y]: = \phi_{x,y}[0,d(x,y)]$ and suppose that, for every $x,y,z\in M$, the following conditions are satisfied.
	\begin{enumerate}[\upshape (i)]
		\item $[x,y] = [y,x]$.
		\item $[x,z]\cap[z,y] = \{z\}\implies z\in[x,y]$.
		\item For every $i\in (0,d(x,y))$, we have $[x,\phi_{x,y}(i)]\subset \phi_{x,y}([0,i])$ and $[\phi_{x,y}(i),y]\subset \phi_{x,y}([i,d(x,y)])$.
	\end{enumerate}
Then $(M,d)$ is $\er$-tree.
\end{lemma}
\begin{proof}Let $\tau:[0,d(x,y)]\to M$ be an isometry with $\tau(0) = x$ and $\tau(d(x,y)) = y$. Fix $i\in(0,d(x,y))$. We will show that $\tau(i) = \phi_{x,y}(i)$.

Put $\sigma_1 := [x,\phi_{x,y}(i)]$, $\sigma_2:= [\phi_{x,y}(i), y]$ and $\rho: = [\phi_{x,y}(i),\tau(i)]$. Then either $\sigma_1\cap \rho = \{\phi_{x,y}(i)\}$ or $\sigma_2\cap \rho = \{\phi_{x,y}(i)\}$. Indeed, fix $u\in \sigma_1\cap \rho$ and $v\in \sigma_2\cap \rho$. By (iii), we have $d(u,v) = d(u,\phi_{x,y}(i)) + d(\phi_{x,y}(i),v)$. Moreover, either $d(\phi_{x,y}(i),u) = d(\phi_{x,y}(i), v) + d(v,u)$ or $d(\phi_{x,y}(i),v) = d(\phi_{x,y}(i),u) + d(u,v)$, depending on how $u$ and $v$ are arranged in $\rho$. It follows that either $u = \phi_{x,y}(i)$ or $v = \phi_{x,y}(i)$.

Let us consider the case when $\sigma_2\cap \rho = \{\phi_{x,y}(i)\}$. Then, by (i) and (ii), we have $\phi_{x,y}(i)\in [\tau(i),y]$; hence, $d(\tau(i),\phi_{x,y}(i)) + d(\phi_{x,y}(i),y) = d(\tau(i),y)$. Since $d(\phi_{x,y}(i),y) = d(x,y) - i = d(\tau(i),y)$, we have $\tau(i) = \phi_{x,y}(i)$. Similarly, if $\sigma_1\cap \rho = \{\phi_{x,y}(i)\}$ then $\tau(i) = \phi_{x,y}(i)$.

As $i\in(0,d(x,y))$ was arbitrary, we have that $\tau = \phi_{x,y}$. Hence, isometries $\phi_{x,y}$ are unique and $(M,d)$ is $\er$-tree.
\end{proof}

\begin{proposition}\label{p:ultrametricInTree}Let $(M,d)$ be a $2^n$-valued ultrametric space. Then there exists an $\er$-tree $(T,\rho)$ such that:
	\begin{enumerate}[\upshape (i)]
		\item $M$ isometrically embeds into $T$.
		\item $\overline{\Br(T)\cup M}$ does not contain any segment $[x,y]$ for $x\neq y$.
		\item $M$ is $4$-Lipschitz retract of $\Br(T)\cup M$.
	\end{enumerate}
Moreover, if $M$ is separable, then $T$ is separable as well.
\end{proposition}
\begin{proof}First, we construct an $\er$-tree $T$ such that $M$ isometrically embeds into $T$. Then we will verify that (ii) and (iii) holds.

Put $Y: = \{(m,i)\setsep m\in M,\,i\in[0,\infty)\}$. Define the following equivalence relation $\sim$ on $Y$:
$$(m,i) \sim (n,j)\quad \Longleftrightarrow i = j\geq d(m,n)/2.$$
Note that $\sim$ is an equivalence relation because $d$ is an ultrametric. Let $\langle m,i\rangle$ denote the equivalence class of $(m,i)$ and put $T := Y/_\sim$. We define the mapping $\rho:T\times T\to [0,\infty)$ by
$$\rho(\langle m,i\rangle, \langle n,j\rangle): = 2\max\{i,j,d(m,n)/2\} - (i+j).
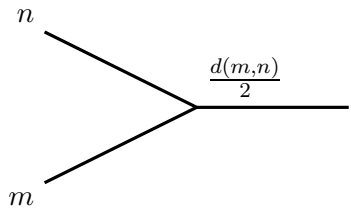
$$
\begin{center}\begin{tikzpicture}
  \coordinate [label={above left:$n$}] (A) at (0,1);
  \coordinate [label={above right:$\frac{d(m,n)}{2}$}] (B) at (2, 0);
  \coordinate [label={below left:$m$}] (C) at (0,-1);
  
  \draw [very thick] (C) -- (B) -- (A);
  \draw [very thick] (B) -- (4,0);  
\end{tikzpicture}\end{center}
Observe that, if $d(m,n)/2 > i$ and $d(m,n)/2 > j$, we have $\rho(\langle m,i\rangle, \langle n,j\rangle) = (d(m,n)/2 - i) + (d(m,n)/2 - j)$. Otherwise, $\rho(\langle m,i\rangle, \langle n,j\rangle) = |i-j|$. It is straightforward to check that $\rho$ is well defined metric on $T$. Obviously, $M\ni m\mapsto \langle m,0\rangle$ is an isometric embedding of $M$ into $T$ and $T$ is separable whenever $M$ is. In order to see that $(T,\rho)$ is $\er$-tree we will find, for every $x,y\in T$, isometry $\phi_{x,y}:[0,\rho(x,y)]\to T$ in such a way that the family $(\phi_{x,y})_{x,y\in T}$ satisfies the assumptions of Lemma \ref{l:tree}.

Fix $x,y\in T$. There are $m,n\in M$ and $i,j\in[0,\infty)$ with $x = \langle m,i\rangle$ and $y = \langle n,j\rangle$. We distinguish the following cases:
	\begin{itemize}
		\item If $j\leq i$ and $i\geq d(m,n)/2$, we put
			$$\phi_{x,y}(t): = \langle n, \rho(x,y) + j - t\rangle,\quad t\in[0,\rho(x,y)].$$
		\item If $j\leq i$ and $i < d(m,n)/2$, we put
			$$\phi_{x,y}(t): = \begin{cases}
			\langle m, i+t\rangle & \text{for } t\in[0,d(m,n)/2 - i]\\
			\langle n, d(m,n) - i - t\rangle & \text{for } t\in[d(m,n)/2 - i, \rho(x,y)].\\
			\end{cases}$$
		\item If $j > i$, we put
			$$\phi_{x,y}(t): = \phi_{y,x}(\rho(x,y) - t),\quad t\in[0,\rho(x,y)].$$
	\end{itemize}
Considering all the possible cases, it is straightforward to check that the family $(\phi_{x,y})_{x,y\in T}$ satisfies the assumptions of Lemma \ref{l:tree}; hence, $(T,d)$ is $\er$-tree and $\phi_{x,y}$ are the unique isometries from the definition of an $\er$-tree.
\begin{claim}\label{claimBr}$\Br(T) = \{\langle m, d(m,n)/2\rangle\setsep m,n\in M, m\neq n\}$
\end{claim}
\begin{proof}``$\supset$'' If $v = \langle m, d(m,n)/2\rangle$ for some $m,n\in M$, then we put $x_1 := \langle m, 0\rangle$, $x_2 := \langle n, 0\rangle$, $x_3 := \langle m, d(m,n)\rangle$ and we check that the points $x_1, x_2, x_3$ are the points witnessing the fact that $v\in\Br(T)$.

``$\subset$'' Fix $v\in\Br(t)$ and let $x_k = \langle m_k, i_k\rangle\in T\setminus\{v\}$, $k\in\{1,2,3\}$ be the points witnessing the fact that $v\in\Br(T)$. There cannot be a point $x\in M$ such that for every $k\in\{1,2,3\}$ we would have $x_k=\langle x,i_k\rangle$, as otherwise, we would have that all the points lie on a common line segment. We distinguish two cases:
\begin{itemize}
\item Two points, let us say $x_1,x_2$, lie on a common branch, i.e. there exists $x\in M$ such that $x_k=\langle x,i_k\rangle$ for $k\in\{1,2\}$. We will show that in this case $v=\langle x,d(x,m_3)/2\rangle$. We may without loss of generality assume that $i_1<i_2$. Notice that $i_1\leq d(x,m_3)/2\leq i_2$. Indeed, if $d(x,m_3)/2<i_1$ then we have $[x_3,x_1]\cap [x_1,x_2] = \{x_1\}$, a contradiction with $x_1\neq v\in [x_3,x_1]\cap [x_1,x_2]$. The case when $i_2<d(x,m_3)/2$ is analogous. It follows that $\langle x,d(x,m_3)/2\rangle$ is a branching point witnessed by $x_1,x_2,x_3$ and since a triple of points can clearly witness at most one branching point it follows that $v=\langle x,d(x,m_3)/2\rangle$.
\item No two points lie on a common branch. By the ultrametric triangle inequality, we may without loss of generality assume that $d(m_1,m_2)\leq d(m_1,m_3) = d(m_3,m_2)$. We claim that $v=\langle m_1,d(m_1,m_2)/2\rangle$. Indeed, it suffices to check that $[x_3,\langle m_1,d(m_1,m_2)/2\rangle ]\cap [x_1,x_2]=\{\langle m_1,d(m_1,m_2)/2\rangle\}$ which follows from the fact that $[x_3,\langle m_1,d(m_1,m_2)/2\rangle ]=[x_3,\langle m_3,d(m_3,m_1)/2]\cup [\langle m_1,d(m_3,m_1)/2\rangle, \langle m_1,d(m_1,m_2)/2\rangle ]$.

\end{itemize}
\end{proof}
Note that so far we have not used the fact that $(M,d)$ is $2^n$-valued. Thus, the embedding as described above works for an arbitrary ultrametric space. In order to prove (ii) and (iii) we will use the assumption that $(M,d)$ is $2^n$-valued. From now on we will not distinguish between $m\in M$ and $\langle m,0\rangle$, its isometric copy in $T$.
\begin{claim}$\Br(T)\cup M = \overline{\Br(T)\cup M}$.
\end{claim}
\begin{proof}
Fix $x = \langle m,i\rangle\in T\setminus (\Br(T)\cup M)$. We need to find $\varepsilon > 0$ with $B(x,\epsilon)\cap (\Br(T)\cup M) = \emptyset$. Find  $n_0$ such that $i > d(m,n_0)/2 = \sup\{d(m,n)/2\setsep n\in M,\; i > d(m,n)/2\}\geq d(m,m)/2 = 0$; note that such an $n_0$ exists, because and the set $\{2^n\setsep n\in\zet\}$ does not have any positive cluster point. If, for every $n\in M\setminus\{m\}$, $i > d(m,n)/2$, we put $\varepsilon : = \min\{i - d(m,n_0)/2,i/2\}$. Otherwise, we find $n_1$ with $i < d(m,n_1)/2 = \inf\{d(m,n)/2\setsep n\in M,\; i < d(m,n)/2\}$ and we put $\varepsilon := \min\{i - d(m,n_0)/2, d(m,n_1)/2 - i, i/2\}$.  In any case straightforward computations show that $B(x,\epsilon) = \{\langle m,j\rangle\setsep |j-i| < \varepsilon\}$ and $B(x,\varepsilon)\cap (\Br(T)\cup M) = \emptyset$.
\end{proof}
As the distances between points in $\Br(T)\cup M$ cannot be irrational numbers, it follows that $\overline{\Br(T)\cup M} = \Br(T)\cup M$ does not contain any segment $[x,y]$ for $x\neq y$. 
Hence, it remains to prove (iii). For every $v\in\Br(T)$ find some $m_v,n_v\in M$ with $v = \langle m_v,d(m_v,n_v)/2\rangle$. We define the retraction $r:\Br(T)\cup M\to M$ as follows:
$$
r(a) : = \begin{cases}a,&\quad\text{ if }a\in M\\
			\langle m_a,0\rangle,&\quad\text{ if }a\in\Br(T).\\
			\end{cases}
$$
Obviously, $r\circ r = r$. It remains to show that $r$ is 4-Lipschitz. If $a,b\in M$ then obviously $\rho(r(a),r(b)) = \rho(a,b)$. Fix $a\in M$ and $b\in\Br(T)$. Then $\rho(r(a),r(b)) = d(a,m_b)$. Hence, the estimation of the Lipschitz constant follows from the following Claim.
\begin{claim}\label{c1}Let $a\in M$ and $b\in\Br(T)$. Then $d(a,m_b) \leq 2\rho(\langle a,0\rangle,b)$.
\end{claim}
\begin{proof}First, let us prove that
	\begin{equation}\label{eq:end}
		d(a,m_b)\leq 2\max\{d(m_b,n_b), d(m_b,a)\} - d(m_b,n_b).
	\end{equation}
Indeed, if $d(m_b,n_b) \geq d(m_b,a)$ we have $$2\max\{d(m_b,n_b), d(m_b,a)\} - d(m_b,n_b) =  d(m_b,n_b)\geq  d(m_b,a).$$
Otherwise, $$2\max\{d(m_b,n_b), d(m_b,a)\} - d(m_b,n_b) = 2 d(m_b,a) - d(m_b,n_b) > d(m_b,a).$$

Now, the following computation proves the Claim:
\begin{equation*}\begin{split}
	d(a,m_b) & \stackrel{\eqref{eq:end}}{\leq}2\max\{d(m_b,n_b), d(m_b,a)\} - d(m_b,n_b)\\
	& = 2\left(2\max\left\{\tfrac{d(m_b,n_b)}{2},\tfrac{d(m_b,a)}{2}\right\} - \tfrac{d(m_b,n_b)}{2}\right) = 2\rho(\langle a,0\rangle,b).
\end{split}\end{equation*}
\end{proof}
Fix $a,b\in\Br(T)$, $a\neq b$. Then $\rho(r(a),r(b)) = d(m_a,m_b)$. Hence, the estimation of the Lipschitz constant follows from the following Claim.
\begin{claim}\label{c2}Let $a,b\in\Br(T)$, $a\neq b$. Then $d(m_a,m_b) \leq 4\rho(a,b)$.
\end{claim}
\begin{proof}As $M$ is $2^n$-valued, there are $m,n,k\in\zet$ with $d(m_a,m_b) = 2^m$, $d(m_a,n_a) = 2^n$ and $d(m_b,n_b) = 2^k$. Interchanging the roles of $a, b$ we may without loss of generality assume that $n\geq k$. Now, we will show that
	\begin{equation}\label{eq:end2}
	2^m\leq 4\left(2^{\max\{m,n\}} - 2^{n-1} - 2^{k-1}\right).
	\end{equation}
Indeed, if $m > n$ we have $1 = 2(1 - 2^{-1})\leq 2(1 - 2^{n-m})$ and
	\begin{equation*}\begin{split}
	2^m& \leq 2^m2\left(1-2^{n-m}\right) = 2\left(2^m - 2^n\right) = 2\left(2^m - 2^{n-1}-2^{n-1}\right) \\
	& \leq 2\left(2^m - 2^{n-1} - 2^{k-1}\right)\leq 4\left(2^{\max\{m,n\}} - 2^{n-1} - 2^{k-1}\right).\\
	\end{split}\end{equation*}
If $m\leq n$ and $n > k$, then we have $1 = 2(1 - 2^{-1})\leq 2(1 - 2^{k-n})$ and
	\begin{equation*}\begin{split}
	2^m & \leq 2^n\leq 2^n2\left(1 - 2^{k-n}\right) = 2\left(2^n - 2^k\right) = 4\left(2^n - 2^{n-1} - 2^{k-1}\right)\\
	& = 4\left(2^{\max\{m,n\}} - 2^{n-1} - 2^{k-1}\right).
	\end{split}\end{equation*}
The remaining case $m\leq n = k$ leads to a contradiction because then we would have $d(m_a,m_b)\leq d(m_b,n_b) = d(m_a,n_a)$, so $(m_b,d(m_b,n_b))\in\langle m_a,d(m_a,n_a)\rangle$ and $b = a$. Thus, \eqref{eq:end2} holds.

Now, the following computation proves the Claim:
\begin{equation*}\begin{split}
d(m_a,m_b) & = 2^m \stackrel{\eqref{eq:end2}}{\leq}4\left(2^{\max\{m,n\}} - 2^{n-1} - 2^{k-1}\right)\\
& = 4\left(\max\left\{2^m,2^n,2^k\right\} - 2^{n-1} - 2^{k-1}\right)\\
& = 4\left(2\max\left\{\tfrac{d(m_a,m_b)}{2},\tfrac{d(m_a,n_a)}{2},\tfrac{d(m_b,n_b)}{2}\right\} - \left(\tfrac{d(m_a,n_a)}{2} + \tfrac{d(m_b,n_b)}{2}\right)\right) = 4\rho(a,b).\\
\end{split}\end{equation*}
\end{proof}
We have verified that $r:\Br(T)\cup M\to M$ is a 4-Lipschitz retraction, which proves (iii). This completes the proof of the Proposition.
\end{proof}

Now, it is straightforward to use the above and prove Theorem \ref{t:freeUltrametric}.

\begin{proof}[Proof of Theorem \ref{t:freeUltrametric}]Let $M$ be a separable ultrametric space. By Fact \ref{f:distances2} and Fact \ref{f:isomorphism}, there is a $2^n$-valued separable ultrametric space $N$ such that $\F(M)$ is isomorphic to $\F(N)$. By Proposition \ref{p:ultrametricInTree} and Fact \ref{f:staciRetrakce}, there is a separable $\er$-tree $T$ such that $\F(N)$ is isometric to a complemented subspace of $\F(\Br(T)\cup N)$ and $\overline{\Br(T)\cup N}$ does not contain any segment. By Proposition \ref{p:godard}, $\F(\Br(T)\cup N)$ is isometric to $\ell_1$. Thus, $\F(M)$ is isomorphic to a complemented subspace of $\ell_1$. It is a well-known result of Pe{\l}czy{\'n}ski, see e.g. \cite[Corollary 4.48]{FHHMZ}, that this is possible only if $\F(M)$ is isomorphic to $\ell_1$.
\end{proof}

\begin{remark}Note that in the proof of Theorem \ref{t:freeUltrametric} we could also use the result of Matou\v{s}ek \cite{mat} to see that $\F(N)$ is isometric to a complemented subspace of $\F(\Br(T)\cup N)$ (because, by \cite{mat}, there is a linear extension operator from $\Lip_0(N)$ to $\Lip_0(T)\supset \Lip_0(\Br(T)\cup N)$). However, we decided to  prove the existence of a retraction instead as it gives us deeper insight into the situation. In this case, the linear extension operator is just $\Lip_0(N)\ni f\mapsto f\circ r$, where $r:\Br(T)\cup N\to N$ is the retraction from (iii) in Proposition \ref{p:ultrametricInTree}.
\end{remark}

\section{Final Remarks and Open problems}

\begin{remark}Note that the use of constant $2$ was important in the proof of Proposition \ref{p:ultrametricInTree}; more precisely, in the proofs of Claims \ref{c1} and \ref{c2}. Therefore, careful examination gives us a rough estimate on the Banach-Mazur distance from $\F(M)$ to $\ell_1$. However, we do not know what is the infimum of all $C > 1$ such that every separable ultrametric space is $C$-isomorphic to $\ell_1$. In the first version of this preprint we asked the following question, which has been already answered - see the remark following this question.
\end{remark}

\begin{question}Is the Lipschitz-free space over a separable ultrametric space isometric to $\ell_1$?
\end{question}

\begin{remark}It has been proved independently by the authors and by A.Dalet, P. Kaufmann and T. Proch\'azka in \cite{aude3} that the Lipschitz-free space over a separable ultrametric space with at least three points is never isometric to $\ell_1$. The proof given by the authors involves some knowledge of properties of $\ell_1$ and certain tedious computations. The proof from \cite{aude3} does not need so many computations and is more related to the properties of $\ell_\infty$. Moreover, it works also in the non separable case. Therefore, it is in our opinion worth reading and we refer the reader there.  Let us sketch our proof here.

Let $N$ be finite subset of a separable ultrametric space $M$. By the inspection of the proof of Theorem \ref{t:schauderUltrametric} we can see that $\F(N)$ is 1-complemented in $\F(M)$ (the reason is that in the proof of Theorem \ref{t:schauderUltrametric} we choose an arbitrary countable dense subset of $M$; hence, we may choose it in such a way that it contains the points from $N$). Now, we recall that every 1-complemented subspace of $\ell_1$ is isometric to $\ell_1$; see e.g. \cite[page 55-56]{liTz} (here the proof is given only for $\ell_p$ spaces with $p\in(1,\infty)$ and it is claimed that in the case of $p=1$ it is simpler - the only difference in the case of $p=1$ is the proof of \cite[Lemma 2.a.5]{liTz}, which is provided e.g. in \cite[Lemma 6.7]{lacey}). Hence, it suffices to prove that the Lipschitz-free space over a three-point ultrametric space is never isometric to $\ell_1^2$.

Let $M = \{x,y,0\}$ be an ultrametric space with metric $d$. Because, for every $r>0$, $\F(M,d)$ is isometric to $\F(M,rd)$ and because Lipschitz-free spaces are isometric if we take another point to be $0$, it suffices to consider the case when $0 < s: = d(x,y)\leq d(x,0) = d(y,0) = 1$. Now, we easily observe that
\[
\begin{split}
\|\delta_x\| = \|\delta_y\| = 1, \\
\|\delta_x - \delta_y\| = s,\\
\|\delta_x + \delta_y\| = 2,\\
\forall \beta > 0:\qquad \max\{s,s\beta,s/2(\beta + 1)\} \leq \|\delta_x - \beta \delta_y\|.\\
\end{split}
\]
Now, let $T:\F(M)\to (\er^2,\|\cdot\|_1)$ be an isometry with $(a_x,b_x) := T(\delta_x)$ and $(a_y,b_y) := T(\delta_y)$. Using the above, the numbers $a_x, b_x, a_y, b_y$ should satisfy the following:
\begin{equation*}
|a_x| + |b_x| = 1,
\end{equation*}
\begin{equation*}
|a_y| + |b_y| = 1,
\end{equation*}
\begin{equation*}
|a_x - a_y| + |b_x - b_y| = s,
\end{equation*}
\begin{equation*}
|a_x + a_y| + |b_x + b_y| = 2,
\end{equation*}
\begin{equation*}
\forall \beta > 0:\qquad \max\{s,s\beta,s/2(\beta + 1)\} \leq |a_x - \beta a_y| + |b_x - \beta b_y|,
\end{equation*}
\begin{equation*}
\forall \beta > 0:\qquad \max\{s,s\beta,s/2(\beta + 1)\} \leq |a_y - \beta a_x| + |b_y - \beta b_x|.
\end{equation*}
However, it is possible to find out that such a system of equations does not have a solution. Even though the computations leading to this conclusion are very tedious, they are absolutely elementary and so we omit them. Moreover, using an alternative approach from \cite{aude3}, it is possible to avoid them.
\end{remark}

It is proved in \cite[Proposition 15.7]{davSem} that every ``uniformly disconnected'' metric space is Lipschitz equivalent to an ultrametric space. We also refer to \cite[Definition 15.1]{davSem} for a precise definition of uniform disconnectedness and for examples. Hence, the following corollary follows from Theorem \ref{t:freeUltrametric} and Fact \ref{f:isomorphism}.
\begin{cor}
Let $M$ be a separable uniformly disconnected metric space.Then $\F(M)$ is isomorphic to $\ell_1$.
\end{cor}
 Every uniformly disconnected is totally disconnected; however, there exists a countable totally disconnected compact space $K$ which is not uniformly disconnected - consider, for example, $K = \{1/n\setsep n\in\en\}\cup\{0\}$. However, A. Dalet proved in \cite{aude1} that $\F(K)$ is a dual space and has MAP whenever $K$ is countable compact. This suggests the following question which has already been asked by G. Godefroy.

\begin{question}Does the Lipschitz-free space over a totally disconnected compact metric space have the BAP? Is it a dual space?
\end{question}

Our last observation concerns linearly rigidity. R. Holmes in \cite{Holmes} proved that the Urysohn universal metric space admits (up to isometry) a unique linearly dense isometric embedding into a Banach space. In other words, any isometric embedding of the Urysohn space with a distinguished point into a Banach space $X$ that maps the distinguished point on $0_X$ extends to a linear isometric embedding of the free space over the Urysohn space $\F (\mathbb{U})$ into $X$. J. Melleray, F. Petrov and A. Vershik in \cite{MePeVe} investigated this property of the Urysohn space further and found other metric spaces with this property. They call them linearly rigid metric spaces. This is another corollary of Theorem \ref{t:freeUltrametric}.

\begin{cor}
No separable ultrametric space is linearly rigid.
\end{cor}
\begin{proof}
By \cite[Theorem 1]{vestfrid}, any separable ultrametric space isometrically embeds into $c_0$. If it were linearly rigid the embedding would extend to an embedding of the free space of this ultrametric space. However, by Theorem \ref{t:freeUltrametric} this free space is isomorphic to $\ell_1$, while it is well known that $c_0$ does not contain a copy of $\ell_1$, a contradiction.
\end{proof}

\subsection*{Acknowledgements}
We would like to thank Eva Perneck\'a for explaining us the general idea of proving the existence of a Schauder basis in a Lipschitz-free space, Aude Dalet for providing us preliminary versions of her papers and Petr H\' ajek for spotting a mistake in the first version of the paper.

The first author was supported by Warsaw Center of Matemathics and Computer Science (KNOW--MNSzW). The second author was supported by funds allocated to the implementation of the international co-funded project in the years 2014-2018, 3038/7.PR/2014/2, and by the EU grant PCOFUND-GA-2012-600415.

\def\cprime{$'$}

\end{document}